\documentclass[letterpaper, 10 pt, conference]{ieeeconf}
\IEEEoverridecommandlockouts
\usepackage[table,usenames,dvipsnames]{xcolor}      
\usepackage[noadjust]{cite}
\usepackage{amsmath,amssymb,amsfonts,amsthm,dsfont,mathtools}
\usepackage{nicematrix}
\usepackage{epstopdf}
\usepackage{adjustbox}
\usepackage{graphicx,tabularx,adjustbox,color}
\usepackage{multirow}
\usepackage[font=footnotesize]{caption}
\usepackage[font=footnotesize]{subcaption}

\usepackage[utf8]{inputenc}

\usepackage{algorithmic}
\usepackage[ruled,vlined]{algorithm2e}
\usepackage{stackengine}

\usepackage{dirtytalk}
\allowdisplaybreaks
\renewcommand{\baselinestretch}{.985}

\usepackage[breaklinks=true, colorlinks, bookmarks=true, citecolor=Black, urlcolor=Violet,linkcolor=Black]{hyperref}

\newtheorem{lemma}{Lemma}
\newtheorem{theorem}{Theorem}
\theoremstyle{definition}
\newtheorem{definition}{Definition}
\newtheorem*{problem*}{Problem}

\newtheorem*{remark*}{Remark}

\newcommand{\ba}{\begin{align}}
\newcommand{\ea}{\end{align}}
\newcommand{\fr}{\frac}

\title{\LARGE \bf Event-Triggered Control of Neuron Growth with Actuation at Soma} 

\author{
Cenk Demir$^1$\thanks{$^1$Department of Mechanical and Aerospace Engineering, UC San Diego, 9500 Gilman Drive, La Jolla, CA, 92093-0411, {\tt\small  cdemir@ucsd.edu, krstic@ucsd.edu}}, 
Shumon Koga$^2$, 
and Miroslav Krstic$^1$ \thanks{$^2$Department of Electrical and Computer Engineering, UC San Diego, 9500 Gilman Drive, La Jolla, CA, 92093-0411, {\tt\small skoga@ucsd.edu}}}

\begin{document}
\maketitle
\begin{abstract}
We introduce a dynamic event-triggering mechanism for regulating the axonal growth of a neuron. We apply boundary actuation at the soma (the part of a neuron that contains the nucleus) and regulate the dynamics of tubulin concentration and axon length. The control law is formulated by applying a Zero-Order Hold (ZOH) to a continuous-time controller which guides the axon to reach the desired length. The proposed dynamic event-triggering mechanism determines the specific time instants at which control inputs are sampled from the continuous-time control law. We establish the existence of a minimum dwell-time between two triggering times that ensures avoidance of Zeno behavior. Through employing the Lyapunov analysis with PDE backstepping, we prove the local stability of the closed-loop system in $L_2$-norm, initially for the target system, and subsequently for the original system. The effectiveness of the proposed method is showcased through numerical simulations.
\end{abstract}

\section{Introduction}

Recent advancements in neuroscience have been achieved from diverse perspectives such as mathematical analysis, physics modeling, and engineering \cite{kandel2000principles,izhikevich2007dynamical}, crucial for understanding neuronal structure and function, and addressing neurological issues. One major challenge in this context is the growth of axons, which are similar to wires and are constructed through the assembly of tubulin proteins. Axons serve as connectors between neurons for transmitting electrical signals. Some neurological diseases, such as Alzheimer's disease \cite{maccioni2001molecular} and spinal cord injuries \cite{liu1997neuronal}, can damage axons by impeding the assembly process of tubulin proteins, leading to halted growth or degeneration. Researchers are developing new therapies to treat these diseases. One promising therapy is called ChABC which involves injecting a bacterial enzyme that digests the axon growth inhibitors \cite{bradbury2011manipulating}. Following this therapy, axon growth can be sustained \cite{karimi2010synergistic}. However, ChABC requires repeated injections due to its rapid inactivation at $37$$^\circ$C \cite{lee2010sustained}. To enhance the effectiveness of this therapy, the amount of enzymes required to achieve the desired axon length and the intervals for these repeated injections must be identified. 


Studying the behavior of tubulin proteins can help achieve the desired axon length. Numerous mathematical models have been proposed for representing the axon growth by Ordinary Differential Equations (ODEs) and Partial Differential Equations (PDEs) to clarify tubulin behavior \cite{mclean2004continuum}. Authors of \cite{diehl2014one} model the axon growth process as a coupled PDE-ODE with a moving boundary, akin to the Stefan problem, effectively describing the associated physical phenomena. In this model, the PDE represents tubulin concentration's evolution along the axon, and the ODEs describe both the evolution of axon length and tubulin concentration at the growth cone.  Given that this model captures this critical information about axon growth, it is worthwhile to consider designing a controller to regulate tubulin concentration and axon length. Over the past two decades, boundary controls of PDE systems have been developed by PDE backstepping, to regulate PDEs by defining control laws at their boundaries \cite{krstic2008boundary}. Following the development of this technique, boundary control was expanded to the class of coupled PDE-ODE systems \cite{krstic09,susto2010control, tang2011state}. While the majority of these contributions are typically assumed to have a constant domain size over time, some recent works focused on a parabolic PDE with a moving domain over time, such as the Stefan problem, which owns a nonlinearity in the moving boundary dynamics. Boundary control for the Stefan problem was achieved in \cite{krstic2020materials}  by backstepping design with global stability results by using maximum principle for parabolic PDE. White, for nonlinear hyperbolic PDEs, several works have proposed local stability results, e.g. \cite{coron2013local}. Our previous works achieved local stability results for nonlinear parabolic PDEs with a moving boundary of the axon growth \cite{demir2021neuroncontrol,demir2022neuron}, and with input delay in \cite{demir2022input}.

While the aforementioned control designs operate in continuous time, certain technologies require control actions only when necessary due to energy, communication, and computation constraints \cite{heemels2012introduction}. To address this, an event-triggered control strategy is proposed for PID controllers in \cite{aaarzen1999simple}, and for state feedback and output feedback controllers for linear and nonlinear time-invariant systems in \cite{heemels2008analysis} and \cite{kofman2006level}. Authors of \cite{postoyan2011unifying}  ensured asymptotic stability for a closed-loop system with state feedback control laws by employing an event-triggering mechanism, characterizing it as a hybrid system. This characterization eased the constraints associated with the event-triggering mechanism and this relaxation is detailed in \cite{girard2014dynamic} as a dynamic triggering approach. In addition, the authors of \cite{espitia2016event} successfully applied an event-triggered mechanism to boundary control hyperbolic PDE systems. This approach led to the use of event-triggered boundary control in reaction-diffusion PDEs, as demonstrated in \cite{espitia2021event}. For Stefan problem, both static and dynamic event-triggered boundary control laws were developed by the authors of \cite{rathnayake2022event2} and \cite{rathnayake2022event}. Furthermore, an event-triggering mechanism was employed to transition between safety utilizing CBFs and stability for Stefan problem with actuator dynamics as discussed in \cite{koga2023event}. In this paper, we introduce a novel dynamic event-triggering mechanism for the axon growth problem which consists of a coupled reaction-diffusion-advection PDE and nonlinear ODEs with a moving boundary. With this dynamic event-triggering mechanism, we aim to address the key question around appropriate time intervals for administering therapy.

The contributions of this paper are (i) designing a control law for neuron growth with actuation at the soma, (ii) developing a dynamic event-triggering mechanism for coupled reaction-diffusion-advection PDEs and nonlinear ODEs with a moving boundary, (iii) analyzing Zeno behavior avoidance, and (iv) demonstrating local stability for the closed-loop system. This work pioneers event-triggering boundary control for axon growth and marks the first local stability analysis using event-triggering mechanisms for PDE systems.


\begin{figure}[t!]
\centering
\includegraphics[width=0.4\linewidth]{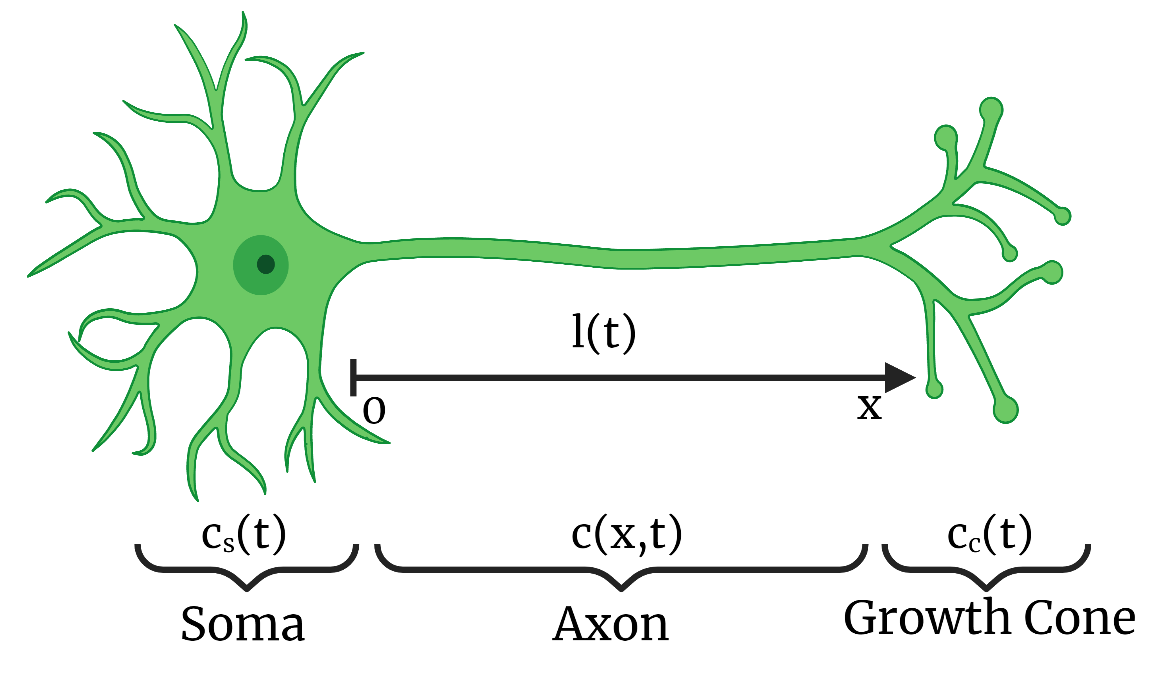}
  \caption{Schematic of neuron and state variables }
  \label{fig:1} 
\end{figure}

\vspace{-0.5em}

\section{Modeling of Axon Growth} \label{sec:model} 
In this section, we introduce a mathematical model for axon growth and provide the steady-state solution for target axon length and a reference error system.

\vspace{-0.5em}

\subsection{Axon growth model by a moving boundary PDE} 
The evolution of tubulin along the axon serves as the primary catalyst for the axon growth process, and to understand this process, we rely on two assumptions to create a mathematical model which are described in our previous work \cite{demir2021neuroncontrol}. Thus, the axonal growth can be modeled as
{
\setlength{\abovedisplayskip}{1pt}
\setlength{\belowdisplayskip}{1pt}
\begin{align}\label{sys1} 
c_t(x,t) =& D c_{xx} (x,t) - a c_x (x,t) - g c(x,t) , \\
\label{sys2} c_x(0,t)+ c(0,t) = &- q_{\rm s}(t), \\
\label{sys3} c(l(t),t) =& c_{\rm c} (t), \\
\label{sys4} l_{\rm c} \dot{c}_{\rm c}(t) = & (a-gl_{\rm c}) c_{\rm c}(t) - D c_x(l(t), t) \notag\\
& - (r_{\rm g} c_{\rm c}(t) + \tilde{r}_{\rm g} l_{\rm c} )(c_{\rm c}(t) - c_{\infty}), \\
\label{sys5} \dot{l}(t) =& r_{\rm g} (c_c(t)-c_{\infty}), 
\end{align}}
In this model, the PDE state $c(x,t)$ represents axonal tubulin concentration. ODE states include $c_{\rm c}(t)$ for tubulin concentration in the growth cone, $l(t)$ for axon length, and $q_{\rm s}(t)$ for tubulin concentration and flux in the soma. Tubulin proteins move along the axon at rate $a$, degrade at rate $g$, and axonal growth halts at equilibrium concentration $c_{\infty}$. The diffusivity constant in \eqref{sys1} is denoted by $D$. Other parameters are detailed in our prior works \cite{demir2021neuroncontrol, demir2022neuron}.

\vspace{-0.5em}

\subsection{Steady-state solution} 
For a desired axon length, $l_s$, we first derive a steady-state solution of the concentration. The steady-state solution of \eqref{sys1}-\eqref{sys5} is obtained as follows
{
\setlength{\abovedisplayskip}{1pt}
\setlength{\belowdisplayskip}{1pt}
\begin{align}
 \label{ceq} 
c_{\rm eq}(x) = c_{\infty} \left( K_{+} e^{\lambda_+ (x - l_{\rm s})} + K_- e^{\lambda_{-} (x - l_{\rm s}) } \right),
\end{align}}
where
{
\setlength{\abovedisplayskip}{1pt}
\setlength{\belowdisplayskip}{1pt}
\begin{small}
\begin{align}
    \lambda_{\pm} =& \frac{a}{2D} \pm \frac{\sqrt{a^2 + 4 D g}}{2 D}, \  K_{\pm} =  \frac{1}{2} \pm  \frac{a  - 2 g l_{\rm c} }{2 \sqrt{a^2 + 4 D g}}.
\end{align}
\end{small}}
The steady-state input for the concentration  in the soma is
{
\setlength{\abovedisplayskip}{1pt}
\setlength{\belowdisplayskip}{1pt}
\begin{small}
\begin{align}
    q_{\rm s}^* = - c_{\infty} \left( K_{+}(1+\lambda_+)  e^{ - \lambda_+ l_{\rm s}} + K_-(1+\lambda_-)  e^{ - \lambda_{-} l_{\rm s} } \right). \label{eqn:state-qs}
\end{align}
\end{small}}

\vspace{-0.5em}

\subsection{Reference error system} 
Let us consider the following reference error states
{
\setlength{\abovedisplayskip}{1pt}
\setlength{\belowdisplayskip}{1pt}
\begin{small}
\begin{align}
u(x,t) = c(x,t) &- c_{\rm eq}(x), \label{eqn:error-usys}\\
z_{1}(t) = c_{\rm c}(t) - c_{\infty}&, \quad
z_2(t) = l(t) - l_{\rm s}, \\
U(t) =  - ( q_{\rm s}&(t) - q_{\rm s}^*). \label{eqn:error-U}
\end{align}
\end{small}}

\noindent where $U(t)$ is the reference error input. Utilizing \eqref{eqn:error-usys}-\eqref{eqn:error-U}, \eqref{ceq} and \eqref{eqn:state-qs} in the governing equations \eqref{sys1}-\eqref{sys5}, we derive the reference error system as
\begin{small}
\begin{align}
&u_t(x,t) = D u_{xx}(x,t) - a u_x(x,t) - g u(x,t) , \label{eqn:u-nonlin}\\
&u_x(0,t)+u (0,t) =  U(t), \label{eqn:non-U}\\
&u(l(t),t) =h(X(t)) , \\
&\dot{X}(t) =  A X(t) + f(X(t)) + B u_x(l(t), t),  \label{eqn:X-nonlin} 
\end{align}
\end{small}
where the state vector, $X(t)\in \mathbb{R}^2$  as $X(t)=[ z_1(t) \quad z_2(t)]^\top$ and constants in \eqref{eqn:u-nonlin}-\eqref{eqn:X-nonlin} are
\begin{small}
\begin{align} 
\label{AB-def} 
 & A = \left[ 
 \begin{array}{cc}
 \tilde a_1 & -\beta \tilde{a}_2 \\
 r_{\rm g} & 0
 \end{array}  
 \right] , ~ B =  \left[ 
 \begin{array}{c}
 - \beta \\
 0
 \end{array}  
 \right], ~ \kappa = \frac{r_{\rm g}}{l_{\rm c}}, ~ \beta =   \frac{D}{l_{\rm c}}, \\
 &f(X(t)) =  - \kappa z_1(t)^2+\beta\tilde{a}_2z_2(t)+\beta  c_{\infty}\frac{a-gl_{\rm c}}{D} \nonumber \\
 &\quad \quad \quad -\beta c_{\infty}\left(  K_{+}\lambda_+ e^{\lambda_+ z_2(t)} + K_- \lambda_{-}e^{\lambda_{-} z_2(t)} \right) \\
 &h(X(t)) =  z_1(t) + c_{\infty}\left(1 - K_{+} e^{\lambda_+ z_2(t)} - K_- e^{\lambda_{-} z_2(t) }\right) , \\
&\tilde a_1 = \frac{a -  r_{\rm g} c_{\infty}}{l_{\rm c}} - g - \tilde{r}_{\rm g} , \quad \tilde{a}_2 =c_{\infty}\left(\lambda_+^2K_++\lambda_-^2K_-\right).
\end{align} 
\end{small}

\vspace{-0.5em}

\section{Continous-time and Sample-based Control Design} \label{sec:control} 
We linearize nonlinear ODEs in \eqref{eqn:X-nonlin} around zero states as

{
\setlength{\abovedisplayskip}{1pt}
\setlength{\belowdisplayskip}{1pt}
\begin{align}
    \label{ulin-PDE}
u_t(x,t) =& D u_{xx}(x,t) - a u(x,t) - g u(x,t) , \\
u_x(0,t)+u (0,t) = & U(t), \label{ulin-BC1} \\
\label{linreferr3}u(l(t),t) =&H^\top X(t)  , \\
\dot{X}(t) = & A_1  X(t) + B u_x (l(t), t), \label{ulin-ODE}
\end{align}}

\noindent where the vector $H \in \mathbb{R}^2$ is defined as
\begin{small}
\begin{align}
  A_1 &= \left[ 
 \begin{array}{cc}
 \tilde a_1 & \tilde a_3 \\
 r_{\rm g} & 0
 \end{array}  
 \right], \color{black}
 ~ H = \left[1 \quad - \frac{(a-gl_{\rm c}) c_{\infty}}{D}\right]^\top ,  \label{C-def}  
\end{align}
\end{small}
where $\tilde{a}_3=\frac{a^2+Dg-agl_{\rm c}}{D^2}$. Our continuous-time control design, detailed in \cite{demir2021neuroncontrol}, employs a backstepping transformation, mapping the linear reference error system $(u,X)$ to the target system $(w,X)$ using the following transformations.
{
\setlength{\abovedisplayskip}{1pt}
\setlength{\belowdisplayskip}{1pt}
\begin{small}
\begin{align}
 \label{bkst}
w(x,t) = & u(x,t) - \int_x^{l(t)} k(x,y) u(y,t) dy  - \phi(x - l(t))^\top X(t), \\
u(x,t)=&w(x,t)+\int_{x}^{l(t)}q(x,y)w(y,t)dy+\varphi(x-l(t))^\top X(t), \nonumber
\end{align} 
\end{small}}

\noindent where $k(.,.)\in \mathbb{R}$, $~q(.,.)\in \mathbb{R} \in \mathbb{R} $, $\phi(.) \in \mathbb{R}^2$ and $\varphi(.) \in \mathbb{R}^2$ are the gain kernel functions are explicitly described in \cite{demir2021neuroncontrol}. The corresponding target system is  
\begin{small}
\begin{align} 
\label{tar-PDE} &w_t(x,t) = D w_{xx} (x,t) - a w_x(x,t) - g w(x,t)- \dot l(t) F(x,X(t)) , \\
\label{tar-BC1} &w_x(0,t)+w(0,t) =  -\frac{1}{D}\left(H-\epsilon\right)^\top Bu(0,t),\\
\label{tar-BC2}&w(l(t),t) =\epsilon^\top X(t) , \\
\label{tar-ODE} &\dot{X}(t) =  (A_1 + BK^\top) X(t) + B w_x(l(t), t),
 \end{align}
 \end{small}

\noindent where $\epsilon\in \mathbb{R}^2$ is chosen in the stability analysis and $K \in \mathbb{R}^2$ is chosen to ensure the stability of $A +BK$ such that it is Hurwitz, satisfying $k_1 > \frac{\tilde a_1}{\beta}$, $k_2 > \frac{\tilde{a}_3}{\beta}$. Furthermore, we describe the redundant nonlinear term $F(x,X(t)) \in \mathbb{R}$ in \eqref{tar-PDE}, arising from the moving boundary, as $F(x,X(t))= \left(\phi'(x-l(t))^T-k(x, l(t)) H^T \right) X(t)$.

\vspace{-0.5em}

 \subsection{Control law} 
The continuous-time control law is obtained based on the boundary condition \eqref{tar-BC1} of the target system at $x = 0$, utilizing the gain kernel solutions as detailed in \cite{demir2021neuroncontrol}.
{
\setlength{\abovedisplayskip}{1pt}
\setlength{\belowdisplayskip}{1pt}
\begin{align}
    \phi(x)^\top&=\begin{bmatrix}(H-\epsilon)^\top & K^\top-\frac{1}{D}H^\top BH^\top\end{bmatrix}e^{N_1x}\begin{bmatrix} I \\ 0
\end{bmatrix},
\label{phix} \\
    &k(x,y)=-\frac{1}{D}\phi(x-y)^\top B,
    \label{kstar}
\end{align}}
where $N_1$ is defined in equation (37) in \cite{demir2021neuroncontrol}.
Substituting $x = 0$ into the transformation \eqref{bkst} yields the control law
{
\setlength{\abovedisplayskip}{1pt}
\setlength{\belowdisplayskip}{1pt}
\begin{align}
  U(t)= -\frac{1}{D}\int_0^{l(t)}p(x)Bu(x,t)dx+p(l(t))X(t),
\label{real-input}
\end{align}}
where $p(x) = \phi'(-x)^\top+\phi(-x)^\top$.
It is worth noting that the solutions of the inverse gain kernels $q(x, y)$ and $\varphi(x)$ can be found in \cite{demir2021neuroncontrol}. This invertibility of the backstepping transformation is essential for demonstrating the stability of the $(u,X)$-system.

\vspace{-0.5em}

\subsection{Sample-based control law} 
We aim to stabilize the closed-loop system \eqref{sys1}-\eqref{sys5} using sampling for the controller defined in \eqref{real-input} with the increasing sequence $(t_j)_{j \in \mathbb{N}}$. The control input is then:
{
\setlength{\abovedisplayskip}{1pt}
\setlength{\belowdisplayskip}{1pt}
\begin{small}
 \begin{align}
     U(t_j)=&-\frac{1}{D}\int_{0}^{l(t_j)}p(x)Bu(x,t_j)dx+p(l(t_j))X(t_j)\label{eqn:dis-Utj}
 \end{align}
 \end{small}
}

\noindent which implies that the boundary condition \eqref{eqn:non-U} is modified 
and therefore, 
the reference error system is rewritten as
{
\setlength{\abovedisplayskip}{1pt}
\setlength{\belowdisplayskip}{1pt}
\begin{align}
   &u_t(x,t) = D u_{xx}(x,t) - a u_x (x,t) - g u(x,t) , \label{eqn:nonlin-u-dis}\\
&u_x(0,t)+u(0,t)=U(t_j), \\
&u(l(t),t) =h(X(t)) , \\
 &\dot{X}(t) =  A X(t) + f(X(t)) + B u_x(l(t) t). \label{eqn:nonlin-X-dis}  
\end{align}}
To establish stability results, we transform the reference error system in \eqref{eqn:nonlin-u-dis}-\eqref{eqn:nonlin-X-dis} to the target system using the transformation in \eqref{bkst}. Thus, the target system is
\begin{small}
\begin{align} 
\label{tar-PDE-nonlin}  &w_t (x,t) = D w_{xx} (x,t) - a w_x (x,t) - g w(x,t) - \dot l(t) F(x,X(t))\notag\\
& \quad \quad \quad -\phi(x-l(t))^\top f(X(t))-G(x,l(t))h^*(X), \\
\label{tar-BC1-nonlin} &w_x(0,t)+w(0,t) = d(t)-\frac{1}{D}\left(H-\epsilon\right)^\top Bu(0,t),\\
\label{tar-BC2-nonlin}&w(l(t),t) =h^*(X(t))+\epsilon^\top X(t) , \\
\label{tar-ODE-nonlin} &\dot{X}(t) =  (A + BK) X(t) +f(X(t)) + B w_x (l(t), t),  
 \end{align}
 \end{small}
 
\noindent where $h^*(X(t))=h(X(t)) -H^\top X(t)$ and the error between continuous-time control law in \eqref{real-input} and sample-based control law in \eqref{eqn:dis-Utj} is defined as
{
\setlength{\abovedisplayskip}{1pt}
\setlength{\belowdisplayskip}{1pt}
\begin{align}
    d(t)&=U(t)-U(t_j).
    \label{eqn:cont-error}
\end{align}}

\noindent and the nonlinear term is
$G(x,l(t)):=\left(\phi'(x-l(t))^\top +\frac{a}{D}\phi(x-l(t))^\top \right) B$. We also apply the following transformation
{
\setlength{\abovedisplayskip}{1pt}
\setlength{\belowdisplayskip}{1pt}
\begin{align*}
    \varpi(x,t)=w(x,t)-h^*(X(t))
\end{align*}}
so this transformation gives us
{
\setlength{\abovedisplayskip}{1pt}
\setlength{\belowdisplayskip}{1pt}
\begin{small}
\begin{align}
\label{eqn:varpi1}
    &\varpi_t(x,t)=D \varpi_{xx} (x,t) - a \varpi_x (x,t) - g \varpi(x,t) \nonumber \\
&\quad \quad  +gh^*(X(t)) - \dot l(t)F(x,X(t))-\dot{h}^*(X(t)) B \varpi_x (l(t), t)  \nonumber \\
&\quad \quad  -\phi(x-l(t))^\top f(X(t)) -G(x,l(t))h^*(X) \nonumber \\
&\quad \quad -\dot{h}^*(X(t))\left((A + BK) X(t) +f(X(t))\right), \\
&\varpi_x(0,t)+\varpi(0,t)=d(t)-\frac{1}{D}\left(H-\epsilon\right)^\top Bu(0,t)+h^*(X(t)), \\
&\varpi(l(t),t)=\epsilon^\top X(t), \\
&\dot{X}(t)=(A + BK) X(t) +f(X(t)) + B \varpi_x (l(t), t)
\label{eqn:varpiend}
\end{align}
\end{small}}

\vspace{-0.5em}

\section{Event-triggered based boundary control} \label{section:event}
In this section, we introduce the event-triggered state-feedback control approach, deriving sampling times for our control law to trigger events.

\vspace{-0.5em}

\begin{definition}
    The design parameters are $\gamma>0$, $\eta>0$, $\rho>0$ and $\beta_i>0$ where $i\in\{1,...5\}$. The event-based controller consists of two trigger mechanisms:
\begin{enumerate}
    \item The event-trigger: The set of all event times are in increasing sequence and they are denoted as $I=\{t_0,t_1,...\}$ where $t_0=0$ with the following rule
    \begin{itemize}
    \item If $S(t,t_{j})=\emptyset$, then the set of the times of the events is $\{t_0,...,t_j\}$.
    \item If $S(t,t_{j})\neq\emptyset$, the next event time is $t_{j+1}=\inf\left(S(t,t_{j})\right)$ where {
\setlength{\abovedisplayskip}{1pt}
\setlength{\belowdisplayskip}{1pt}\begin{small}\begin{align}
    S(t,t_{j})=\{t\in \mathbb{R}_+|t>t_j\wedge d^2(t)>-\gamma m(t)\} 
    \end{align}
    \end{small}}

\noindent for all $t\in[t_j,t_{j+1})$, $d(t)$ is given by \eqref{eqn:cont-error} and $m(t)$ satisfies the ODE
{
\setlength{\abovedisplayskip}{1pt}
\setlength{\belowdisplayskip}{1pt}
\begin{small}
\begin{align}
    &\dot{m}(t)=-\eta m(t)+\rho d(t)^2-\beta_1 X(t)^2-\beta_2X(t)^4\nonumber \\
    &-\beta_3X(t)^6-\beta_4|w(0,t)|^2-\beta_5 ||w(x,t))||^2 \label{eqn:m-def}
\end{align}
\end{small}}
        \end{itemize}
        \item The control action: The feedback control law that is derived in \eqref{eqn:dis-Utj} for all $t\in [t_j,t_{j+1})$ where $j\in \mathbb{N}$.
\end{enumerate}
\end{definition}

\vspace{-0.5em}

\begin{lemma}
    Under the definition of the state feedback event-triggered boundary control, it holds that 
    $d^2(t)\leq -\gamma m(t)$ and $m(t)>0$ for $t\in [0,F)$, where $F=\sup(I)$.
\end{lemma}
\begin{proof}
The proof closely resembles that of Lemma 2 in \cite{rathnayake2021observer}, and thus omitted for brevity. 
\end{proof}

\vspace{-1em}

\begin{lemma}
     For all $t\in(t_j,t_{j+1})$ where $j\in\mathbb{N}$, it holds that
     {
\setlength{\abovedisplayskip}{1pt}
\setlength{\belowdisplayskip}{1pt}
     \begin{small}
    \begin{align}
        (\dot{d}(t))^2\leq& \rho_1d^2(t)+\alpha_1X(t)^2+\alpha_2X(t)^4+\alpha_3X(t)^6\nonumber \\
        &+\alpha_4w(0,t)^2+\alpha_5||w(x,t)||^2,
    \end{align}
    \end{small}}
for the following positive constants and defined functions: 
{
\setlength{\abovedisplayskip}{1pt}
\setlength{\belowdisplayskip}{1pt}
    \begin{small}
\begin{align}
    \rho_1&=7|p(0)B|^2, \\
    \alpha_1&=\frac{21}{2}\left|\frac{1}{D}\zeta(y)B\int_0^{l(t)}\varphi(x-l(t))^\top dx\right|^2+28(p(0)Bp(l(t)))^2\nonumber \\
    +&21\left(\left(p(0)\left(1-\frac{a}{D}\right)+\dot{p}(0)\right)B\right)^2(\varphi(0)^\top)^2+28(p(l(t))A)^2\nonumber \\
    +&14\left(\left|\dot{p}(l(t))+\frac{a}{D}p(l(t))+\frac{r_{\rm g}}{D}e_1p(l(t))\right|BH^\top\right)^2, \label{eqn:alpha1}\\
    \alpha_2&=7\left(r_{\rm g}e_1 \dot{p}(l(t))+2k_n\left|\dot{p}(l(t))+\frac{a}{D}p(l(t))\right|B+p(l(t))\kappa\right)^2\\
    \alpha_3&=7\left(\frac{2k_n}{D}r_{\rm g}e_1p(l(t))B \right)^2+28\left(k_mp(l(t))\right)^2, \\
    \alpha_4&=21\left(\left(p(0)\left(1-\frac{a}{D}\right)+\dot{p}(0)\right)B\right)^2,\\
    \alpha_5&=7\left|\frac{1}{D}\zeta(y)B\right|^2\left(\frac{9}{2}+\frac{9}{2}\left(\int_{0}^{l(t)}\int_{x}^{l(t)}q(x,y)^2dydx\right)\right) \nonumber \\
    +&21\left(\left(p(0)\left(1-\frac{a}{D}\right)+\dot{p}(0)\right)B\right)^2\bar{G}(l(t))^2, \label{eqn:alpha5} \\
     &\zeta(y):=\int_0^{l(t)}D\ddot{p}(y)-a\dot{p}(y)+gp(y)-p(0)Bp(y)dy, \\
     &\bar{G}(l(t)):=\int_0^{l(t)}q(0,x)dx
\end{align}
\end{small}}
\end{lemma}
\begin{proof}
The proof closely resembles that of Lemma 2 in \cite{rathnayake2021observer}, and thus it is omitted.
\end{proof}

\vspace{-1.5em}

\section{Main Results}
In this section, we present the analysis for the avoidance of Zeno behavior and closed-loop system stability.

\vspace{-0.5em}

\subsection{Avoidance of Zeno Behavior}
The event-triggering mechanism dictates when to sample the continuous-time control signal, reducing computational and communication complexity. However, defining these sampling times is challenging due to the potential for Zeno behavior, where specific instances may result in infinite triggering within finite time intervals, limiting the mechanism's applicability. To address this, we prove the existence of a minimum dwell-time in the following theorem.

\vspace{-0.5em}

\begin{theorem}
Consider the closed-loop system of \eqref{sys1}-\eqref{sys5} incorporating the control law given by \eqref{eqn:dis-Utj} and the triggering mechanism in Definition 1. There exists a minimum dwell-time denoted as $\tau$ between two consecutive triggering times $t_j$ and $t_{j+1}$, satisfying $t_{j+1}-t_j\geq \tau$ for all $j\in\mathbb{N}$ when $\beta_i$  is selected as follows:
{
\setlength{\abovedisplayskip}{1pt}
\setlength{\belowdisplayskip}{1pt}
\begin{align}
    \beta_i=\frac{\alpha_i}{\gamma(1-\sigma)}
    \label{def-beta-i}
\end{align}}
where $\sigma\in(0,1)$, $i=\{1,...,5\}$ and the values of $\alpha_i$ are provided in equations \eqref{eqn:alpha1}-\eqref{eqn:alpha5}.
\end{theorem}

\vspace{-0.5em}

\begin{proof}
By using Lemma 1, we define the continuous function $\psi(t)$ in $[t_j,t_{j+1})$ to derive the lower bound between interexecution times as follows:
{
\setlength{\abovedisplayskip}{1pt}
\setlength{\belowdisplayskip}{1pt}
\begin{align}
\psi(t):=\frac{d^2(t)+\gamma(1-\sigma)m(t)}{-\gamma \sigma m(t)}
\label{def-psi}
\end{align}}
As described in \cite{Espitia2017}, one can show that 
{
\setlength{\abovedisplayskip}{1pt}
\setlength{\belowdisplayskip}{1pt}
\begin{align}
    \dot{m}(t)=&-\eta m(t)+\rho d(t)^2-\beta_1 X(t)^2-\beta_2X(t)^4\nonumber \\
    &-\beta_3X(t)^6-\beta_4|w(0,t)|^2-\beta_5 ||w||^2 
\end{align}}
Taking the time derivative of $\eqref{def-psi}$ and using Lemma 1, we can choose $\beta_i$ as described in \eqref{def-beta-i}.
Thus, we get $\dot{\psi}(t)\leq a_1 \psi(t)^2+a_2\psi(t)+a_3$, where 
$a_1=\rho\sigma\gamma>0$, $a_2=1+2\rho_1+(1-\sigma)\rho+\eta>0$ and $a_3=(1+\rho_1+\gamma(1-\sigma)\rho+\eta)\frac{1-\sigma}{\sigma}>0$. Using the comparison principle and the argument in \cite{Espitia2017}, one can prove that there exists a time minimum dwell-time $\tau$ as follows:
{
\setlength{\abovedisplayskip}{1pt}
\setlength{\belowdisplayskip}{1pt}
\begin{align}
    \tau=\int_{0}^{1}\frac{1}{a_1s^2+a_2s+a_3}ds
\end{align}}
which completes the proof.
\end{proof}

\vspace{-0.5em}

\subsection{Stability Analysis}
In this section, we initially introduce the main theorem, which establishes stability.

\vspace{-0.5em}

\begin{theorem}
    Consider the closed-loop system comprising the plant described by \eqref{sys1}-\eqref{sys5} along with the control law specified by \eqref{eqn:dis-Utj} and employing an event-triggering mechanism that is defined in Definition 1. Let 
{
\setlength{\abovedisplayskip}{1pt}
\setlength{\belowdisplayskip}{1pt}
    \begin{equation}
\rho\geq \frac{d_1^2D}{\delta_1} \label{eqn:def-gamma}
    \end{equation}}
    and $\eta>0$ be design parameters, $\sigma\in(0,1)$ while $\beta_i$ for $i=\{1,2,3,4,5\}$ are chosen 
  as in \eqref{eqn:alpha1}-\eqref{eqn:alpha5}. 
 Then, there exist constants $M>0$, $c>0$ and $\Gamma$, such that, if initial conditions is such that $Z(0)<M$ then the following norm estimate is satisfied:
 {
\setlength{\abovedisplayskip}{1pt}
\setlength{\belowdisplayskip}{1pt}
    \begin{align}
        Z(t)\leq cZ(0)exp(-\Gamma t),
    \end{align}}
    for all $t\geq0$, in $L_2$-norm 
$Z(t)=||u(.,t)||_{L_2}^2+X^\top X$ which establishes the local exponential stability of the origin of the closed-loop system.
\end{theorem}

\vspace{-0.5em}

To establish local stability on a non-constant spatial interval, we rely on two system properties from \cite{demir2021neuroncontrol}, outlined below:
{
\setlength{\abovedisplayskip}{1pt}
\setlength{\belowdisplayskip}{1pt}
    \begin{align}
     0 < l(t) \leq \bar l, \quad
    |\dot l(t) | \leq \bar v, \label{ineq-ldot}  
\end{align}}
for some $\bar l>l_{\rm s} >0$ and $\bar v>0$. Then, we consider the following Lyapunov functionals
{
\setlength{\abovedisplayskip}{1pt}
\setlength{\belowdisplayskip}{1pt}
\begin{align}
\label{V1-def}    V_1 =& \fr{1}{2} ||\varpi||^2 := \fr{1}{2} \int_0^{l(t)} \varpi(x,t)^2 dx, \\
V_2 =& X(t)^\top P_1 X(t),  \quad V_3=\frac{1}{2}X(t)^\top P_2X(t)
    \label{V3-def} 
\end{align}}
where $P_1\succ0$ and $P_2\succeq 0$ are positive definite and positive semidefinite matrices satisfying the Lyapunov equations: 
{
\setlength{\abovedisplayskip}{1pt}
\setlength{\belowdisplayskip}{1pt}\begin{align}
    &(A + BK^\top )^\top P_1 + P_1 (A + BK^\top ) = - Q_1, \nonumber \\
    &(A + BK^\top )^\top (P_1+P_2) + (P_1+P_2) (A + BK^\top ) = - Q_2 \nonumber
\end{align}}
where
{
\setlength{\abovedisplayskip}{1pt}
\setlength{\belowdisplayskip}{1pt}
\begin{align}
    P_1=\begin{bmatrix}
        p_{1,1} & p_{1,2} \\
        p_{1,2} & p_{2,2} 
    \end{bmatrix}, \quad P_2=\begin{bmatrix} \frac{D\epsilon_1}{\beta}-2p_{1,1} &0 \\ 0 & 0\end{bmatrix}
\end{align}}
where we pick $\epsilon\in\mathbb{R}^2$ as 
$\epsilon_1\geq 2l_{\rm c}p_{1,1}$ and  $\epsilon_2 = \frac{p_{1,2}}{l_{\rm c}d_1}$ for some positive definite matrices $Q_1\succ0$ and $Q_2\succ0$. We define the total Lyapunov function as
{
\setlength{\abovedisplayskip}{1pt}
\setlength{\belowdisplayskip}{1pt}
\begin{align}
    V(t) = d_1 V_1(t) + V_2(t)  + d_2 V_3(t)-m(t), 
    \label{Vtotal}
\end{align}}
where $d_1>0$ and $d_2>0$ are parameters to be determined.

\vspace{-0.5em}

\begin{lemma}
Assume that the conditions in \eqref{ineq-ldot} are satisfied with $\bar v=\frac{D}{16(D+1)}$, for all $t\geq 0$. Then, for sufficiently large $d_1>0$ and sufficiently small $d_2<0$, there exist positive constants $\xi_i$ for $i=\{1,2,3,4,5\}$ such that the following norm estimate holds for $t\in (t_j,t_{j+1})$, $j\in\mathbb{N}$:
{
\setlength{\abovedisplayskip}{1pt}
\setlength{\belowdisplayskip}{1pt}
\begin{small}
    \begin{align}
        \dot{V}\leq -\alpha^*V+\sum_{i=1}^4\xi_iV^{(1+\frac{i}{2})}
    \end{align}
\end{small}}
where $\alpha^*=\min\left\{\frac{g}{2},\frac{1}{2\lambda_{\min}(P_1+P_2)},\eta\right\}$.
\end{lemma}

\begin{proof}
We take the time derivative of the Lyapunov functional \eqref{V1-def}-\eqref{V3-def} along the target system, substituted boundary conditions for $t\in(t_j+t_{j+1})$, $j\in\mathbb{N}$ with
{
\setlength{\abovedisplayskip}{1pt}
\setlength{\belowdisplayskip}{1pt}
\begin{align}
    \varpi_x(l(t),t)=\bar{B}\left(\dot{X}-(A+BK)X(t)-f(X(t))\right)
\end{align}}
where $\bar{B}=[-\beta^{-1}~ 0]$. Then, applying Poincaré's, Agmon's, and Young's inequalities, \eqref{eqn:def-gamma}, along with \eqref{eqn:m-def}, and using the following inequalities 
{
\setlength{\abovedisplayskip}{1pt}
\setlength{\belowdisplayskip}{1pt}
\begin{align} 
   \left|h(X)\right| &\leq 2k_n X^\top X+|H^\top X|, \label{def-boundH}\\
   f(X(t)) &\leq \kappa X^\top X+2k_m|X^\top X|^{3/2},
\end{align}}

\noindent where $k_n=c_{\infty}\max\{K_{+}\lambda_{+}^2,K_{-}\lambda_{-}^2\}$ and $ ~ 
    k_m=c_{\infty}\max\{K_{+}\lambda_{+}^3,K_{-}\lambda_{-}^3\}$
by utilizing $-e^x+x+1\leq x^2$ for $x\leq 1.79$, the expression for \eqref{Vtotal} can be transformed into:
{
\setlength{\abovedisplayskip}{1pt}
\setlength{\belowdisplayskip}{1pt}
\begin{small}
\begin{align}
     \dot{V}\leq &-\alpha^* V +\xi_1 V^{3/2}+\xi_2 V^2+\xi_3V^{5/2}+\xi_4V^{3}
    \label{def-dotVtot3}
\end{align}
\end{small}}
where
{
\setlength{\abovedisplayskip}{1pt}
\setlength{\belowdisplayskip}{1pt}
\begin{small}
\begin{align}
    \xi_1&=\frac{\left(Dd_1|\epsilon \bar{B}|+2d_2\left|P_1^\top B\bar{B}\right|\right)\kappa^2+\frac{d_1r_{\rm g}}{2}(1+L_1)+r_{\rm g}}{d_2^{3/2}\lambda_{\min}(P_1+P_2)^{3/2}} \label{def:Xi1-new}\\
    \xi_2&=\frac{\Xi_1}{d_2^{2}\lambda_{\min}(P_1+P_2)^2}, \quad 
    \xi_3=\frac{4d_2k_m|P_1|}{d_2^{5/2}\lambda_{\min}(P_1+P_2)^{5/2}},\\
    \xi_4&=\frac{\Xi_2}{d_2^3\lambda_{\min}(P_1+P_2)^3},\label{def:Xi4-new}
\end{align}
\end{small}}

\noindent taking into account $\dot{l}(t)=r_{\rm g}e_2^\top X$ and choosing the constants $d_1$ and $d_2$ to satisfy
{
\setlength{\abovedisplayskip}{1pt}
\setlength{\belowdisplayskip}{1pt}
\begin{small}
\begin{align}
    d_1&\geq \max\left\{\frac{8\bar{l}\left(D+2\right)+16\bar{l}\beta_4}{D},\frac{4\beta_5+7}{g}\right\}, \\
    d_2&\geq \frac{4}{\lambda_{\min}(Q_2)}\left(Dd_1\left|\epsilon \bar{B}|(A+BK)\right|+\beta_1\right) \nonumber \\
    &+\frac{4}{\lambda_{\min}(Q_2)}\left(\left(D+2+Dd_1+\frac{d_1a}{2}+2\beta_4\right)\frac{2}{\beta^2}\right).
\end{align}
\end{small}}
%
Note that the positive constants in \eqref{def:Xi1-new}-\eqref{def:Xi4-new} are given as
{
\setlength{\abovedisplayskip}{1pt}
\setlength{\belowdisplayskip}{1pt}
\begin{small}
\begin{align}
    &F(0,X(t))^2 \leq L_1 X^\top X, \quad \int_0^{l(t)}\left(\phi(x-l(t))^\top\right)^2dx  \leq  L_{n_2}, \nonumber
    \\
    &\int_0^{l(t)}\left(\phi'(x-l(t))^\top B -ak(x,l(t))\right)^2dx  \leq L_{n_3}
\label{eqn:Ln3}
    \\
    &\Xi_1=4Dd_1|\epsilon \bar{B}|k_m^2|P_1|^2+8d_2\left|P_1^\top B\bar{B}\right|k_m^2|P_1|^2+\beta_2\nonumber \\
    &\quad +2d_1^2L_{n_3}k_n^2+\frac{d_1^2}{2}L_{n_2}\kappa^2+d_1^2c_{\infty}^2r_{\rm g}^2k_l+8d_2\kappa |P|\beta_5 k_{l}\ \nonumber\\
    &\quad+2d_2\kappa |P|\left(d_1^2\left(\beta^2(1-\epsilon_1)^2\left(1+\bar{G}(l(t))^2\right)+D\right)\right)k_l
    \\
    &\Xi_2=d_1^2c_{\infty}^2r_{\rm g}^2k_l+\frac{d_1^2}{2}L_{n_2}4k_m^2|P_1|^2+\beta_3 \nonumber \\
    &\quad+\left(d_1^2\left(\beta^2(1-\epsilon_1)^2\left(1+\bar{G}(l(t))^2\right)+D\right)+4\beta_5\right)k_l 
    \\
    &k_l=\max\left\{\left|K_+\lambda_+\right|,\left|K_-\lambda_-\right|\right\}^2
\end{align}
\end{small}}
which completes the proof of Lemma 3.
\end{proof}

\vspace{-0.5em}

In this next section, we ensure the local stability of the closed-loop system with the event-triggering mechanism. 

\vspace{-0.5em}

\begin{lemma}
    In the region $\Omega_1:=\{(\varpi,X) \in L_2 \times \mathbb{R}^2 | V(t) < M_0\}$ where $t\in(t_j,t_{j+1})$ $j\in\mathbb{N}$, there exists a positive constant $M_0>0$ such that the conditions in \eqref{ineq-ldot} hold.
\end{lemma}

\vspace{-0.5em}

\begin{proof}
    See the proof of Lemma 2 in \cite{demir2021neuroncontrol}.  
\end{proof}

\vspace{-0.5em}

From the proof of Lemma 4, we have $M_0=\frac{\lambda_{\rm min}(P)}{d_2}r^2$ for $t\in (t_j,t_{j+1})$, $j\in\mathbb{N}$. Next, we analyze stability within the time interval $t \in (t_j, t_{j+1})$ for $j \in \mathbb{N}$, and subsequently for $t\in(0,t)$. Within this interval, we establish the following lemma:

\vspace{-0.5em}

\begin{lemma}
    There exists a positive constant $M_j$ such that if $V(t_j)<M_j$ then the following norm estimate holds for $t \in (t_j, t_{j+1})$, where $j \in \mathbb{N}$:
    {
\setlength{\abovedisplayskip}{1pt}
\setlength{\belowdisplayskip}{1pt}
    \begin{align}
        V(t_{j+1})\leq V(t_j)e^{-\frac{\alpha^*}{2}(t_{j+1}-t_j)}
    \end{align}}
\end{lemma}

\begin{figure}[t!]
\centering
\subfloat[{The axon length, $l(t)$ successfully converges to the desired length by $t=4.5$ mins for both event-triggered and continuous-time control law.}]{       \includegraphics[width=0.95\linewidth]{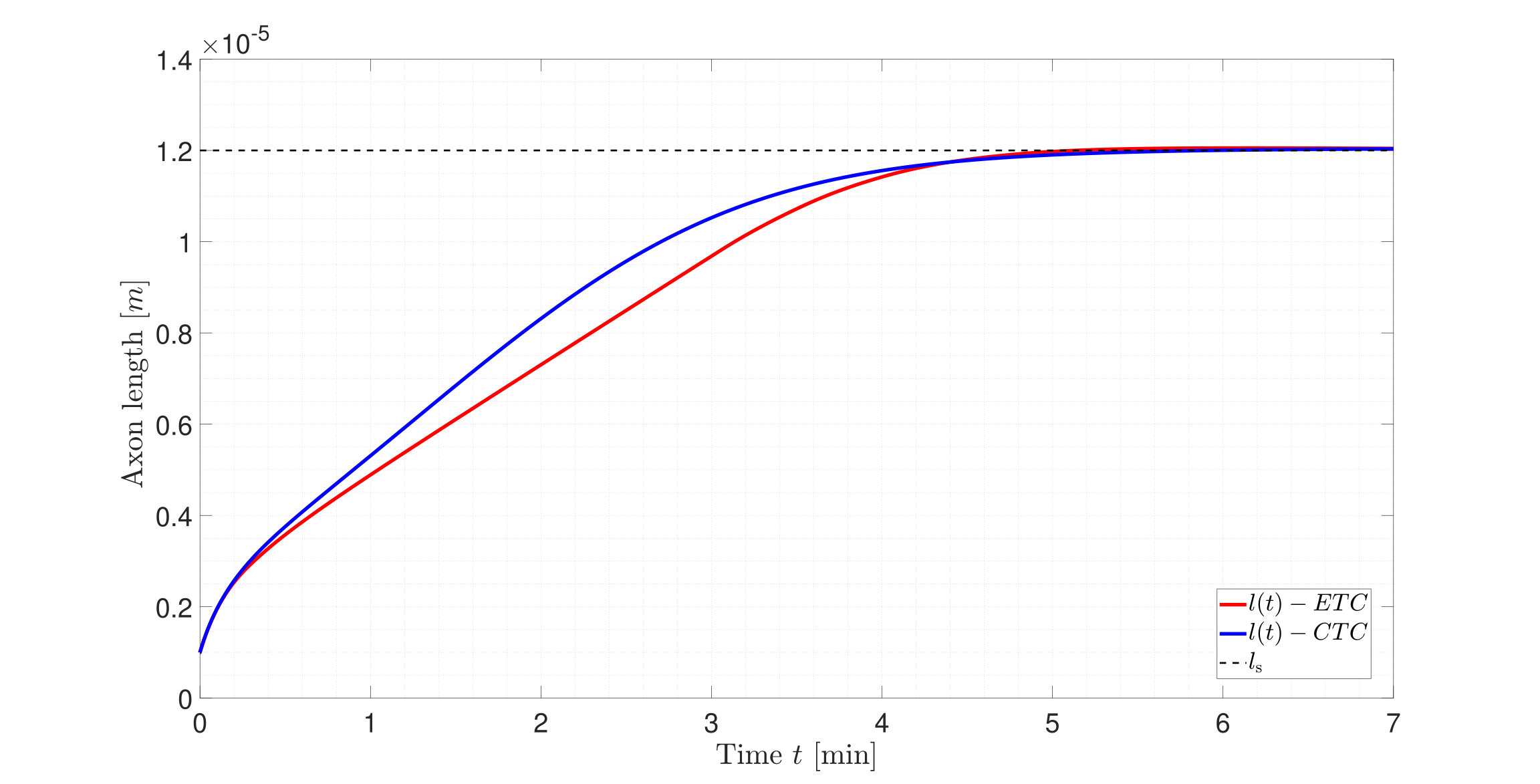}}
  \caption{The closed-loop response of the continuous-time and event-triggered control law for $l_s=12\mu m$ }
  \label{fig:2} 
\end{figure}

\vspace{-0.75em}

\begin{proof}

    {
\setlength{\abovedisplayskip}{1pt}
\setlength{\belowdisplayskip}{1pt}
\begin{table}[!b]
\hfill	\caption{\label{tab:initial}Biological constants and control parameters}
	\centering
	\begin{tabular}{c|c|c|c} 
		\hline
		Parameter  & Value  & Parameter & Value \\
		\hline
		$D$ & $10\times10^{-12}  m^2/s$& $\tilde{r}_{\rm g}$ & $0.053$\\
		$a$& $1\times 10^{-8}  m/s$ & $\gamma$ &  $10^4$\\
		$g$& $5\times 10^{-7} \ s^{-1}$ & $l_{\rm c}$ & $4\mu m$\\
		$r_{\rm g}$& $1.783\times 10^{-5} \ m^4/(mol s)$ & $l_s$ & $12\mu m$ \\
		$c_{\infty}$ &  $0.0119  \ mol/m^3$ & $l_0$& $1\mu m$ \\
 \hline
	\end{tabular}
\end{table}
}
For $M_j>0$, we easily demonstrate that $M_j<M_0$ using Lemma 4, ensuring the norm estimate from Lemma 3 holds. Thus, we set $M_j\leq p^*$, where $p^*$ is a non-zero root of the polynomial for $V>0$.
{
\setlength{\abovedisplayskip}{1pt}
\setlength{\belowdisplayskip}{1pt}
  \begin{align}
      -\alpha^* V +\xi_1 V^{3/2}+\xi_2 V^2+\xi_3V^{5/2}+\xi_4V^{3}=0 
      \label{eqn:poly}
  \end{align}}
Since $\alpha^*$, and $\xi_i$ are all positive, at least one positive root exists for the polynomial in \eqref{eqn:poly}. Therefore, \eqref{def-dotVtot3} implies 
  {
\setlength{\abovedisplayskip}{1pt}
\setlength{\belowdisplayskip}{1pt}
  \begin{align}
      \dot{V}(t)\leq -\frac{\alpha^*}{2}V(t)
  \end{align}}
for $t \in (t_j, t_{j+1})$, $j \in \mathbb{N}$ where $M_j=\min\left\{M_0,p^*\right\}$. The continuity of $V(t)$ in this interval implies $V(t_{j+1}^-)=V(t)$ and $V(t_{j}^+)=V(t_j)$  where $t_{j}^+$ and $t_{j}^-$ are right and left limits of $t=t_j$, respectively. Thus, we have
{
\setlength{\abovedisplayskip}{1pt}
\setlength{\belowdisplayskip}{1pt}
\begin{align}
    V(t_{j+1})\leq \exp(-\alpha^*(t_{j+1}-t_{j}))V(t_j)
\end{align}}
which completes the proof of Lemma 5.
\end{proof}

\vspace{-0.5em}

For any $t\geq 0$ in $t\in[t_{j},t_{j+1})$, $j\in\mathbb{N}$, we obtain
{
\setlength{\abovedisplayskip}{1pt}
\setlength{\belowdisplayskip}{1pt}
\begin{align}
    V(t)\leq e^{-\alpha^*(t-t_j)}V(t_j) 
    \leq e^{-\alpha^*t}V(0)
\end{align}}
Recalling $m(t)<0$ and \eqref{Vtotal}, we have
{
\setlength{\abovedisplayskip}{1pt}
\setlength{\belowdisplayskip}{1pt}
\begin{align}
    d_1 V_1(t) + V_2(t)  + d_2 V_3(t)\leq e^{-\alpha^*t}V(0).
\end{align}}

Utilizing the norm equivalence principle between the $(\varpi, X)$ system and the $(w, X)$ system, and leveraging the invertibility of the backstepping transformation, we establish the local exponential stability of $(u, X)$ in the $L_2$-norm.

\vspace{-0.5em}

\section{Numerical Simulations}

\vspace{-0.5em}

\begin{figure}[!t]
    \centering
    \subfloat[{Event-triggered control }\label{fig:1a}]{%
\includegraphics[width=0.95\linewidth]{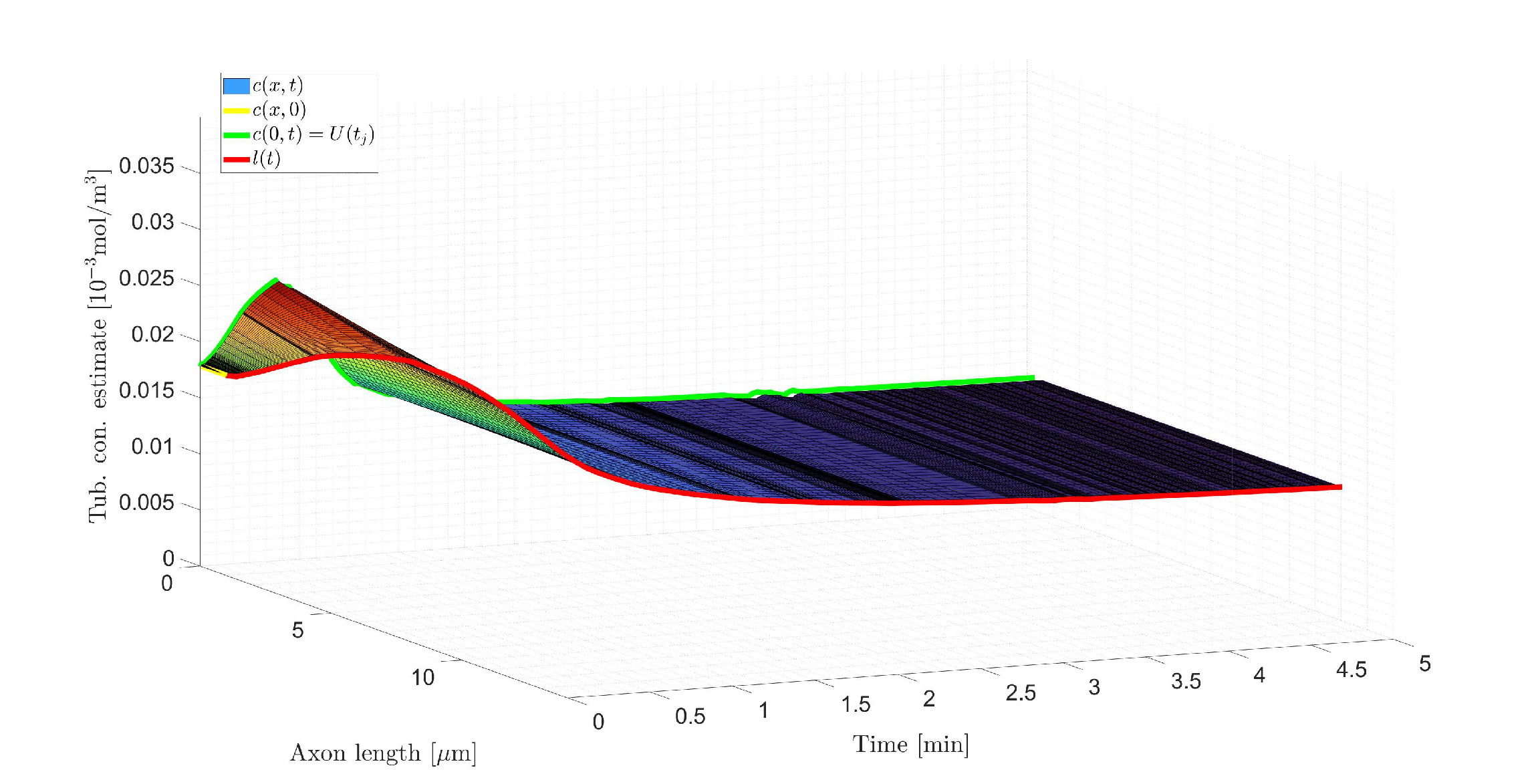}}
    \vfill
\subfloat[{Continuous-time control}\label{fig:1b}]{ \includegraphics[width=0.95\linewidth]{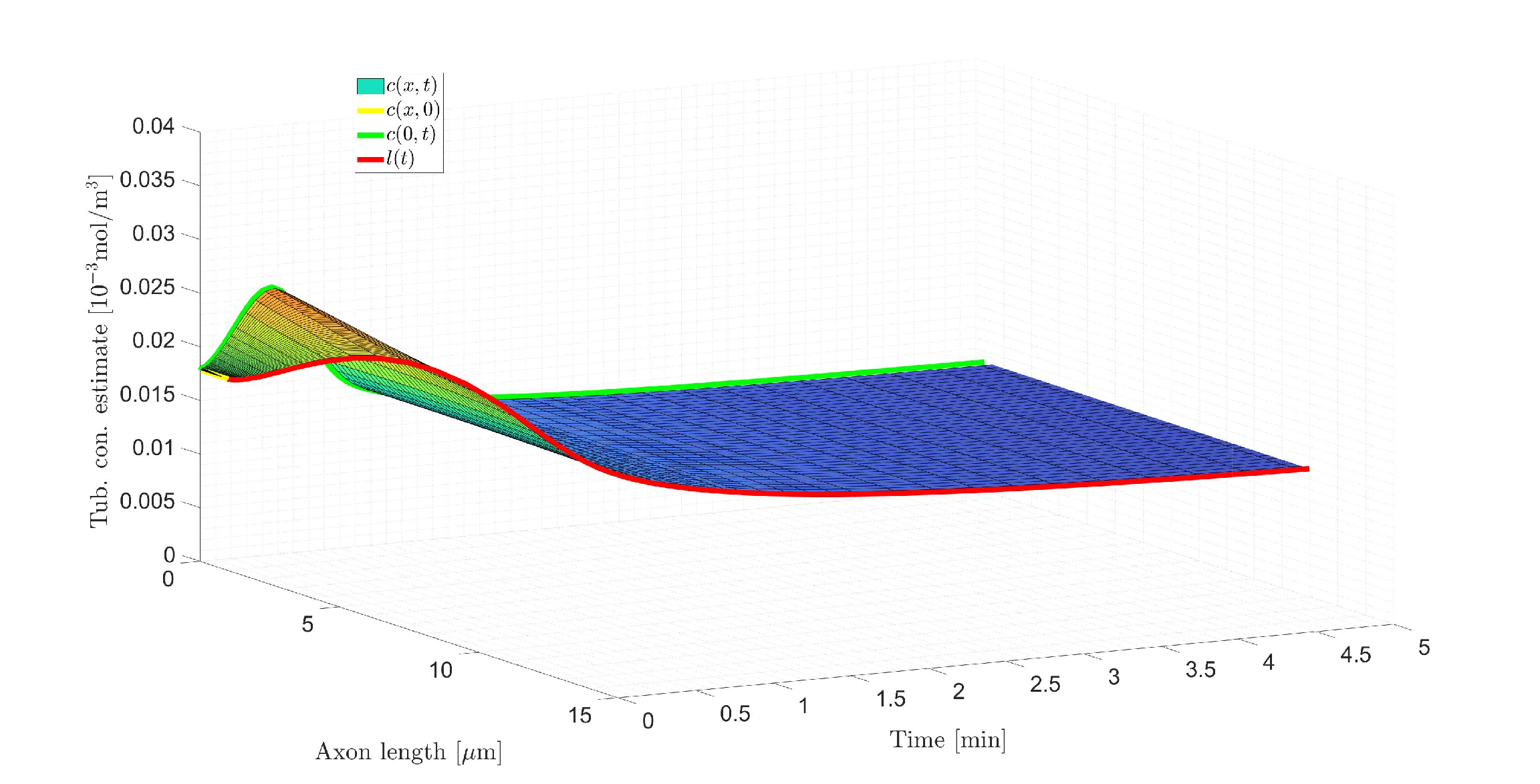}
}
    \caption{The closed-loop response of the designed full-state feedback control system for continuous-time and event-triggered control law. }
\end{figure}

In this section, we numerically analyze the plant dynamics \eqref{sys1}-\eqref{sys5} using the control law \eqref{real-input} and the event triggering mechanism as in Section \ref{section:event}. The model employs biological constants and control parameters from Table 1, with initial conditions set to $c_0(x)=1.5c_{\infty}$ for the tubulin concentration along the axon and $l_0=1\mu m$ for the initial axon length. The control gain parameters are chosen as $k_1=-0.001$ and $k_2=3\times 10^{13}$. The event-triggering mechanism parameters are set as follows: $m(0)=-0.5$, $\beta_1=4.0849\times 10^{8}$, $\beta_2=1.307\times10^{10}$, $\beta_3=1.642\times 10^{11}$, $\beta_4=6.536\times 10^{11}$, $\beta_5=7.35\times 10^{11}$, $\rho=4\times 10^{22}$, $\eta=800$ and $\sigma=0.5$.  In Fig. \ref{fig:1a} and \ref{fig:1b}, we present the evolution of tubulin concentration along the axon for both continuous-time control law and event-triggered control. Fig. 2 shows axon growth convergence under continuous-time and event-triggered control laws. Both methods achieve the desired $12\mu m$ length from an initial $1\mu m$ in about 4.5 minutes. 



\vspace{-0.5em}

\section{Conclusion}

\vspace{-0.5em}

This paper explores a dynamic event-triggering boundary control approach for axonal growth modeling, addressing Zeno behavior avoidance and providing a local stability analysis of the closed-loop system. Future research will focus on periodic event-triggering and self-triggering boundary control methods, better suited for digital implementations.

\vspace{-0.75em}

\bibliographystyle{IEEEtranS}
\bibliography{main.bib}

\begin{thebibliography}{10}
\providecommand{\url}[1]{#1}
\csname url@samestyle\endcsname
\providecommand{\newblock}{\relax}
\providecommand{\bibinfo}[2]{#2}
\providecommand{\BIBentrySTDinterwordspacing}{\spaceskip=0pt\relax}
\providecommand{\BIBentryALTinterwordstretchfactor}{4}
\providecommand{\BIBentryALTinterwordspacing}{\spaceskip=\fontdimen2\font plus
\BIBentryALTinterwordstretchfactor\fontdimen3\font minus \fontdimen4\font\relax}
\providecommand{\BIBforeignlanguage}[2]{{%
\expandafter\ifx\csname l@#1\endcsname\relax
\typeout{** WARNING: IEEEtranS.bst: No hyphenation pattern has been}%
\typeout{** loaded for the language `#1'. Using the pattern for}%
\typeout{** the default language instead.}%
\else
\language=\csname l@#1\endcsname
\fi
#2}}
\providecommand{\BIBdecl}{\relax}
\BIBdecl

\bibitem{aaarzen1999simple}
K.-E. {\AA}arz{\'e}n, ``A simple event-based pid controller,'' \emph{IFAC Proceedings Volumes}, vol.~32, no.~2, pp. 8687--8692, 1999.

\bibitem{bradbury2011manipulating}
E.~J. Bradbury and L.~M. Carter, ``Manipulating the glial scar: chondroitinase abc as a therapy for spinal cord injury,'' \emph{Brain research bulletin}, vol.~84, no. 4-5, pp. 306--316, 2011.

\bibitem{coron2013local}
J.-M. Coron, R.~Vazquez, M.~Krstic, and G.~Bastin, ``Local exponential h\^{}2 stabilization of a 2$\backslash$times2 quasilinear hyperbolic system using backstepping,'' \emph{SIAM Journal on Control and Optimization}, vol.~51, no.~3, pp. 2005--2035, 2013.

\bibitem{demir2021neuroncontrol}
C.~Demir, S.~Koga, and M.~Krstic, ``Neuron growth control by pde backstepping: Axon length regulation by tubulin flux actuation in soma,'' in \emph{2021 60th IEEE Conference on Decision and Control (CDC)}, 2021, pp. 649--654.

\bibitem{demir2022input}
------, ``Input delay compensation for neuron growth by pde backstepping,'' \emph{IFAC-PapersOnLine}, vol.~55, no.~36, pp. 49--54, 2022.

\bibitem{demir2022neuron}
------, ``Neuron growth output-feedback control by pde backstepping,'' in \emph{2022 American Control Conference (ACC)}.\hskip 1em plus 0.5em minus 0.4em\relax IEEE, 2022, pp. 4159--4164.

\bibitem{diehl2014one}
S.~Diehl, E.~Henningsson, A.~Heyden, and S.~Perna, ``A one-dimensional moving-boundary model for tubulin-driven axonal growth,'' \emph{Journal of theoretical biology}, vol. 358, pp. 194--207, 2014.

\bibitem{espitia2016event}
N.~Espitia, A.~Girard, N.~Marchand, and C.~Prieur, ``Event-based control of linear hyperbolic systems of conservation laws,'' \emph{Automatica}, vol.~70, pp. 275--287, 2016.

\bibitem{espitia2021event}
N.~Espitia, I.~Karafyllis, and M.~Krstic, ``Event-triggered boundary control of constant-parameter reaction--diffusion pdes: A small-gain approach,'' \emph{Automatica}, vol. 128, p. 109562, 2021.

\bibitem{Espitia2017}
N.~Espitia, A.~Girard, N.~Marchand, and C.~Prieur, ``Event-based boundary control of a linear $2\times 2$ hyperbolic system via backstepping approach,'' \emph{IEEE Transactions on Automatic Control}, vol.~63, no.~8, pp. 2686--2693, 2018.

\bibitem{girard2014dynamic}
A.~Girard, ``Dynamic triggering mechanisms for event-triggered control,'' \emph{IEEE Transactions on Automatic Control}, vol.~60, no.~7, pp. 1992--1997, 2014.

\bibitem{heemels2012introduction}
W.~P. Heemels, K.~H. Johansson, and P.~Tabuada, ``An introduction to event-triggered and self-triggered control,'' in \emph{2012 ieee 51st ieee conference on decision and control (cdc)}.\hskip 1em plus 0.5em minus 0.4em\relax IEEE, 2012, pp. 3270--3285.

\bibitem{heemels2008analysis}
W.~Heemels, J.~Sandee, and P.~Van Den~Bosch, ``Analysis of event-driven controllers for linear systems,'' \emph{International journal of control}, vol.~81, no.~4, pp. 571--590, 2008.

\bibitem{izhikevich2007dynamical}
E.~M. Izhikevich, \emph{Dynamical systems in neuroscience}.\hskip 1em plus 0.5em minus 0.4em\relax MIT press, 2007.

\bibitem{kandel2000principles}
E.~R. Kandel, J.~H. Schwartz, T.~M. Jessell, S.~Siegelbaum, A.~J. Hudspeth, and S.~Mack, \emph{Principles of neural science}.\hskip 1em plus 0.5em minus 0.4em\relax McGraw-hill New York, 2000, vol.~4.

\bibitem{karimi2010synergistic}
S.~Karimi-Abdolrezaee, E.~Eftekharpour, J.~Wang, D.~Schut, and M.~G. Fehlings, ``Synergistic effects of transplanted adult neural stem/progenitor cells, chondroitinase, and growth factors promote functional repair and plasticity of the chronically injured spinal cord,'' \emph{Journal of Neuroscience}, vol.~30, no.~5, pp. 1657--1676, 2010.

\bibitem{kofman2006level}
E.~Kofman and J.~H. Braslavsky, ``Level crossing sampling in feedback stabilization under data-rate constraints,'' in \emph{Proceedings of the 45th IEEE Conference on Decision and Control}.\hskip 1em plus 0.5em minus 0.4em\relax IEEE, 2006, pp. 4423--4428.

\bibitem{koga2023event}
S.~Koga, C.~Demir, and M.~Krstic, ``Event-triggered safe stabilizing boundary control for the stefan pde system with actuator dynamics,'' in \emph{2023 American Control Conference (ACC)}.\hskip 1em plus 0.5em minus 0.4em\relax IEEE, 2023, pp. 1794--1799.

\bibitem{krstic2020materials}
S.~Koga and M.~Krstic, \emph{Materials Phase Change PDE Control and Estimation: From Additive Manufacturing to Polar Ice}.\hskip 1em plus 0.5em minus 0.4em\relax Springer Nature, 2020.

\bibitem{krstic09}
M.~Krstic, ``Compensating actuator and sensor dynamics governed by diffusion {PDE}s,'' \emph{Systems \& Control Letters}, vol.~58, no.~5, pp. 372--377, 2009.

\bibitem{krstic2008boundary}
M.~Krstic and A.~Smyshlyaev, \emph{Boundary control of PDEs: A course on backstepping designs}.\hskip 1em plus 0.5em minus 0.4em\relax SIAM, 2008.

\bibitem{lee2010sustained}
H.~Lee, R.~J. McKeon, and R.~V. Bellamkonda, ``Sustained delivery of thermostabilized chabc enhances axonal sprouting and functional recovery after spinal cord injury,'' \emph{Proceedings of the National Academy of Sciences}, vol. 107, no.~8, pp. 3340--3345, 2010.

\bibitem{liu1997neuronal}
X.~Z. Liu, X.~M. Xu, R.~Hu, C.~Du, S.~X. Zhang, J.~W. McDonald, H.~X. Dong, Y.~J. Wu, G.~S. Fan, M.~F. Jacquin \emph{et~al.}, ``Neuronal and glial apoptosis after traumatic spinal cord injury,'' \emph{Journal of Neuroscience}, vol.~17, no.~14, pp. 5395--5406, 1997.

\bibitem{maccioni2001molecular}
R.~B. Maccioni, J.~P. Mu{\~n}oz, and L.~Barbeito, ``The molecular bases of alzheimer's disease and other neurodegenerative disorders,'' \emph{Archives of medical research}, vol.~32, no.~5, pp. 367--381, 2001.

\bibitem{mclean2004continuum}
D.~R. McLean, A.~van Ooyen, and B.~P. Graham, ``Continuum model for tubulin-driven neurite elongation,'' \emph{Neurocomp.}, vol.~58, pp. 511--516, 2004.

\bibitem{postoyan2011unifying}
R.~Postoyan, A.~Anta, D.~Ne{\v{s}}i{\'c}, and P.~Tabuada, ``A unifying lyapunov-based framework for the event-triggered control of nonlinear systems,'' in \emph{2011 50th IEEE conference on decision and control and European control conference}.\hskip 1em plus 0.5em minus 0.4em\relax IEEE, 2011, pp. 2559--2564.

\bibitem{rathnayake2022event2}
B.~Rathnayake and M.~Diagne, ``Event-based boundary control of one-phase stefan problem: A static triggering approach,'' in \emph{2022 American Control Conference (ACC)}.\hskip 1em plus 0.5em minus 0.4em\relax IEEE, 2022, pp. 2403--2408.

\bibitem{rathnayake2022event}
------, ``Event-based boundary control of the stefan problem: A dynamic triggering approach,'' in \emph{2022 IEEE 61st Conference on Decision and Control (CDC)}.\hskip 1em plus 0.5em minus 0.4em\relax IEEE, 2022, pp. 415--420.

\bibitem{rathnayake2021observer}
B.~Rathnayake, M.~Diagne, N.~Espitia, and I.~Karafyllis, ``Observer-based event-triggered boundary control of a class of reaction--diffusion pdes,'' \emph{IEEE Transactions on Automatic Control}, vol.~67, no.~6, pp. 2905--2917, 2021.

\bibitem{susto2010control}
G.~A. Susto and M.~Krstic, ``Control of pde--ode cascades with neumann interconnections,'' \emph{Journal of the Franklin Institute}, vol. 347, no.~1, pp. 284--314, 2010.

\bibitem{tang2011state}
S.~Tang and C.~Xie, ``State and output feedback boundary control for a coupled pde--ode system,'' \emph{Syst. Contr. Lett.}, vol.~60, no.~8, pp. 540--545, 2011.

\end{thebibliography}


\end{document}